\documentclass[12pt]{article}
\usepackage[usenames]{color}
\usepackage{graphicx, subfigure}
\usepackage{amsmath, amsthm, amssymb}
\usepackage{amsfonts}
\usepackage{fullpage}
\usepackage{ifthen}
\usepackage{url}
\usepackage[sort&compress]{natbib}
\usepackage{multirow}
\usepackage{bm}
\usepackage{tikz-qtree}
\usepackage[linesnumbered,algoruled,boxed,lined,commentsnumbered]{algorithm2e}
\usepackage[section]{placeins}
\usepackage{soul}
\usepackage{xcolor,cancel}
\usepackage{mathtools}

\newtheorem{theorem}{Theorem}[section]

\makeatletter
\def\@biblabel#1{}
\makeatother

\theoremstyle{plain}
\newtheorem{cor}{Corollary}

\theoremstyle{definition}
\newtheorem{example}{Example}

\theoremstyle{remark}

\title{On the true detection probability of the \\
uniformly optimal search plan}

\author{Liang Hong\footnote{Liang Hong is a Professor in the Department of Mathematical Sciences,  The University of Texas at Dallas, 800 West Campbell Road, Richardson, TX 75080, USA. Tel.:~972-883-2161. Email address: liang.hong@utdallas.edu.}}
\date{May 29, 2025}

\usepackage{tikz}
\usetikzlibrary{automata,arrows,positioning,calc}
\usepackage{listings}
\usepackage[labelformat=empty]{caption}
\usepackage{subfig}
\usepackage{caption}

\begin{document}

\maketitle

\begin{abstract}
The gold standard for designing a search plan is to select a target distribution and then find the uniformly optimal search plan based on it.  This approach has been successfully applied in several high-profile civil and military search missions.  Since the target distribution is subjective and chosen at an analyst's discretion,  it is natural to ask whether this approach can generate a search plan that maximizes the true detection probability at each moment.  This article gives a negative answer by establishing that, under mild conditions,  for a given target distribution and the uniformly optimal search plan based on it, there is another target distribution whose induced uniformly optimal search plan leads to an increased true detection probability at every moment.  In particular, it implies that the problem of finding a search plan that maximizes the true detection probability at each moment remains unsolved. 

\smallskip

\end{abstract}

\section{Introduction}
The optimal search theory studies the best way of searching for a target given a limited budget when the target location is uncertain.  It originated in the US Navy when its Anti-Submarine Warfare Operations Research Group needed to find a way to efficiently detect hostile submarines during World War II (e.g., Koopman 1945, 1956a, b, c).   Later progress of the theory was largely driven by real-world maritime search missions too; see, for instance, Stone and Stanshine (1971), Stone (1973, 1975, 1976), Richardson and Stone (1971),   Richardson et al. (1980),  Stone et al. (2014a),  Vermeulen and Brink (2017), and Bourque (2019).  Most recently, the game-theoretic approach to the search problem has received considerable attention; see, for instance,  Alpern and Gal (2003),   Clarkson et al. (2020),  Lidbetter (2020),  Alpern et al. (2021),  and Lin (2021).  Though we have witnessed incredible advancement in naval technology since World War II,  the optimal search theory remains relevant and important,  as evidenced by recent notable naval accidents, such as the loss of Indonesia's KRI Nanggala submarine in April 2021,  the collision of the USS Connecticut with a seamount in the South China Sea in October 2021,  and the crash of an F-35C of the US Navy in the South China Sea January 2022.  Moreover,  the optimal search theory can also be applied to space exploration endeavors, such as NASA's Artemis campaign and Mars Missions. 

The gold standard for designing a search plan is based on the Bayesian approach: to choose a target distribution and then obtain the uniformly optimal search plan based on it.  This approach has been successfully employed in several high-profile civil and military search missions; see Richardson and Stone (1971), Richardson et al. (1980), Stone (1992), and Stone et al. (2014a), and references therein.  This approach also forms the theoretical underpinnings of the U.S. Coast Guard's Search and Rescue Optimal Planning System (SAROPS) and its predecessor Computer Assisted Search Planning (CASP); see, for example,  Stone (1975) and Kratzke et al. (2010).  The uniformly optimal search plan plays a key role in this approach. Arkin (1964) first established sufficient conditions for a uniformly optimal search plan to exist in the Euclidean search space.   Stone (1973, 1975, 1976) extended and generalized his work.  Properties of the uniformly optimal search plan have been widely studied (e.g., Stone 1975, Stone et al. 2016, Hong 2024).  By its definition, the uniformly optimal search plan maximizes the detection probability at each moment.  Since the target distribution is subjective (hence there is no true target distribution), the detection probability a uniformly optimal search plan maximizes is subjective. It is not the true/objective detection probability which only depends on the true target location,  the (objective) detection function, and the amount of effort allocated to the true target location.  Given this observation, it is natural to ask an important question:
\begin{enumerate}
\item[]\emph{Can an analyst find a target distribution whose induced uniformly optimal search plan maximizes the true detection probability at each moment?}
\end{enumerate}

\noindent The key contribution of this article is to give a negative answer. Specifically,   it shows that, under certain conditions,  for any chosen target distribution and the optimal search plan based on it,  there is a different target distribution whose induced uniformly optimal search plan results in a larger true detection probability at every moment.

This result has several interesting consequences. First,  this result demonstrates that the aforementioned gold standard will not lead to a genuinely optimal solution to the search problem.  Moreover,  this result serves as a caveat to practitioners: a uniformly optimal search plan maximizes the subjective detection probability at every moment, but it does not necessarily maximize the true detection probability at every moment.  Finally,  this result shows that the search problem is profoundly deep and challenging.  The problem of finding a search plan that maximizes the true detection probability at each moment remains unsolved.  

The remainder of the article is organized as follows.  In Section~2, we establish our notational conventions by reviewing the search problem and the uniformly optimal search plan.  Next, in Section~3, we provide two motivating examples to build intuition.  Then, we establish our main results in Section~4; there we treat the discrete case and the continuous case separately.  Finally,  in Section 5,  we conclude the article with a discussion on the societal and practical relevance of our main finding.

\section{Notation and setup}

We are interested in searching for a stationary target where the exact target location $x$ is unknown,  the equipment used for detection, such as a sonar or radar, is not perfect, and the amount of effort available is limited.  Though the exact target location is unknown,  we often have some information about it. To avail ourselves of this information and, at the same time, quantify the uncertainty of the target location, we first pin down an area, in our assessment,  that definitely contains the target.  This area is called the \emph{possibility area} and denoted as $X$.  Then we construct a \emph{target distribution} for $x$ that is supported on $X$.  We will use $\Pi$ to denote the (cumulative) distribution function of the proposed target distribution.  That is, the possibility area $X$ is the support of $\Pi$.  $X$ can be either countable (i.e., finite or countably infinite) or uncountable.  If $X$ is countable,  the smallest region over which search effort can be allocated is represented by a cell; we refer to this case as the \emph{discrete case}.  Without loss of generality, we may number these cells using positive integers.  Also, $X$ is always finite in any real-world application.  Therefore,  we always take $X$ to be a finite subset of $\{1, 2, \ldots\}$ in the discrete case.  If $X$ is uncountable, then it is assumed that the search effort can be allocated in a continuous manner; this case will be referred to as the \emph{continuous case}.  In the continuous case,  we assume $X$ is a subset of the $n$-dimensional Euclidean space $\mathbb{R}^n$.  Throughout, we assume  $\Pi$ admits a  density function $\pi$.  That is, $\pi$ denotes the probability mass function (i.e.,  the discrete density function) in the discrete case and the probability density function in the continuous case.  We have implicitly taken the Bayesian approach because in most realistic cases, such as a marine accident,  the scenario will never repeat itself under identical conditions (i.e., the sample size of the available data is $1$), rendering a frequentist approach inappropriate.  Note that the target distribution is subjective and there is no true target distribution.  Since the true target location is uncertain in the search problem,  the target distribution is always assumed to be non-degenerate. 

Once the possibility area $X$ is nailed down, we must decide how to allocate the effort in $X$.  Let $\mathbb{R}_+$ denote the non-negative real line $[0, \infty)$.   In the discrete case,  we define an \emph{allocation} on $X$ to be a function $f: X\rightarrow \mathbb{R}_+$; for the continuous case,  we define an \emph{allocation} on $X$ to be a function $f: X\rightarrow \mathbb{R}_+$ such that $\int_Af(x)dx$ equals the amount of search effort allocated in $A$ where $A$ is any subset of $X$. Therefore, the total effort invested in the discrete case and the continuous case are $\sum_{x\in X}f(x)$ and $\int_Xf(x)dx$ respectively. 

In the search problem, there are two sources of uncertainty: (i) uncertainty of the target location, (ii) uncertainty of detection due to imperfection of the sensor.  The former is quantified by a target distribution.  To account for the latter, we will use a \emph{detection function} $d: X\times \mathbb{R}_+\rightarrow [0, 1]$. In the discrete case, $d(x, y)$ is the conditional probability of detecting the target when $y$ amount of effort is placed in cell $x$ given that the target is in cell $x$.  For the continuous case,   $d(x, y)$ denotes the conditional probability of detecting the target if the effort density equals $y$ at $x$ given that the target is at $x$.  Different from the target distribution, the detection function is objective and depends on the sensor.  Throughout, we assume the detection function has been either correctly derived from physical laws or reliably estimated from repeated testing.  

Given a target distribution $\Pi$,  the allocation $f$, and the detection function $d$,  the \emph{subjective probability of detection},  denoted as $P[f]$, equals
\[
P[f]=\left\{
		                           \begin{array}{ll}
		                           \sum_{x\in X}d(x, f(x))\pi(x),& \hbox{in the discrete case,} \\
					\int_Xd(x, f(x))\pi(x)dx, & \hbox{in the continuous case.} 
		                          \end{array}
		                         \right.
\]
Since $\pi$ is subjective,  $P[f]$ is not the true probability of detection.  Let $x_0$ denote the true cell that contains the target  in the discrete case and the true target location in the continuous case. Then
the \emph{true/objective probability of detection},  denoted as $P^{\#} [f]$, equals
\begin{equation}
\label{eq:trueprob}
P^{\#}[f]=d(x_0, f(x_0)).
\end{equation}

Every search has a budget.  To account for this, we first define a cost function $c: X \times \mathbb{R}_+ \rightarrow \mathbb{R}_+$.  In the discrete case,  $c(x,  y)$ denotes the cost of applying $y$ effort in cell $x$; in the continuous case, $c(x, y)$ stands for the cost density of applying effort density $y$ at location $x$. For an allocation $f$ on $X$, the cost resulting from $f$, denoted as $C[f]$, is given by
\[
C[f]=\left\{
		                           \begin{array}{ll}
		                           \sum_{x\in X} c(x, f(x)),& \hbox{in the discrete case,} \\
					\int_Xc(x, f(x))dx, & \hbox{in the continuous case.} 
		                          \end{array}
		                         \right.
\]
Unless otherwise stated,  we assume the cost is proportional to allocation, i.e.,  $c(x, y)=y$ for all $x\in X$. 

A \emph{search plan} on $X$ is a function $\varphi: X\times \mathbb{R}_+\rightarrow \mathbb{R}_+$ such that
\begin{enumerate}
\item[(i)]$\varphi(\cdot, t)$ is an allocation on $X$ for all $t\geq 0$;
\item[(ii)]$\varphi(x, \cdot)$ is an increasing function for all $x\in X$.
\end{enumerate}
Condition~(i)  ensures that $\varphi(\cdot, t)$ is an allocation in the first $t$ units of time.  Condition~(ii) says that no effort can be undone as the search progresses.  Let $T$ be the time to find the target using $\varphi$.  We will use $\mu(\varphi)$ to denote the expectation of $T$ with respect to the subjective probability of detection, i.e., the subjective mean time to find the target. It is well-known that (e.g., Lemma~3.4 of Kallenberg 2002)
\begin{equation*}
\mu(\varphi)=\int_0^\infty  (1-P[\varphi(\cdot, t)])dt.
\end{equation*}
Similarly,  we will use $\mu^{\#}$ to denote the expectation of $T$ with respect to the true probability of detection, i.e., the true mean time to find the target. We have
\begin{equation}
\label{eq:truemean}
\mu^{\#} (\varphi)=\int_0^\infty(1-P^{\#}[\varphi(\cdot, t)])dt.
\end{equation}

We define the \emph{cumulative effort function} $E$ to be a non-negative function with domain $\mathbb{R}_+$ such that $E(t)$ is the effort available by time $t$.  We assume $E$ is increasing and $E(t)>0$ for all $t>0$.  There are many different search plans available to a search team.  We need a criterion for comparing two distinct search plans. Ideally,   it would be the true probability of detection.  Since the true target location is unknown, the true probability of detection cannot be calculated in actuality.  Hence,  the best we can do is to use the subjective probability of detection.  Given a target distribution $\Pi$,  the \emph{uniformly optimal search plan for $\Pi$ and $E(t)$}  maximizes the subjective probability of detection at each time $t\geq 0$, subject to the constrain $C[\varphi^\star(\cdot, t)]\leq E(t)$.  To make a precise definition, let $\Phi(E)$ be the class of search plans $\varphi$ such that 
\begin{equation}
\label{eq:uniformdis1}
\sum_{x\in X}\varphi(x, t)=E(t), \quad \text{for all $t\geq0$},
\end{equation}
for the discrete case,  and 
\begin{equation}
\label{eq:uniformcon1}
\int_X\varphi(x, t)dx=E(t), \quad \text{for all $t\geq0$},
\end{equation}
for the continuous case. A search plan $\varphi^\star\in \Phi(E)$ is said to be \emph{uniformly optimal for $\Pi$ within $\Phi(E)$} if 
\begin{equation}
\label{eq:uniform}
P[\varphi^\star(\cdot, t)]=\max\{P[\varphi(\cdot, t)\mid \varphi\in \Phi(E)]\}\quad \text{for all $t\geq 0$},
\end{equation}
where $P[\varphi^\star]$ is the subjective probability of detection based on $\Pi$.   A uniformly optimal search plan has several desirable properties; see Chapter~3 of Stone (1975) and Hong (2024).  However, a uniformly optimal search plan does not always exist, as Example~2.2.9 of Stone (1975) shows.   But we have the following sufficient conditions for the existence of uniformly optimal search plans; see, for instance, Section~2.4 of Stone (1975).  

\begin{theorem}
\label{thm:unifexistence}
\
\begin{enumerate}
\item[(i)]If $\Pi$ is a target distribution on a discrete possibility area $X$, $d(x, 0)=0$,  and $d(x, \cdot)$ is continuous,  concave, and increasing for each $x\in X$,  there exists a uniformly optimal search plan for $\Pi$ within $\Phi(E)$.   
\item[(ii)]If $\Pi$ is a target distribution on a continuous possibility area $X$,  $d(x, 0)=0$, and $d(x, \cdot)$ is increasing and right-continuous for each $x\in X$,  there exists a uniformly optimal search plan for $\Pi$  within $\Phi(E)$.   
\end{enumerate}
\end{theorem}

A class of detection functions,  called \emph{regular functions}, satisfy these sufficient conditions.  A detection function $d$ is said to be \emph{regular} if $d(x, 0)=0$ and $\partial d(x, y)/\partial y$ is continuous, positive, and strictly decreasing for all $x\in X$.  A detection function $d$ is said to be \emph{homogeneous} if $d(x_1, y)=d(x_2,y)$ for all $x_1, x_2\in X$ and $y\geq 0$.  One of the most frequently used detection functions is the \emph{exponential function}:
\[d(x, y)=1-e^{-\alpha_xy}, \quad  y\geq 0,\]
where  $\alpha_x>0$  is called the \emph{rate} of $d$ at $x\in X$. A \emph{homogeneous and exponential} detection function  with rate $\alpha>0$ takes the form
 \[d(x, y)=1-e^{-\alpha y}, \quad y\geq 0.\]
When the detection function is regular,  a semi-closed form of the uniformly optimal search plan can be derived under mild conditions; see Chapter~2 of Stone (1975) or Chapter~5 of Washburn (2014).  In particular, we have the following theorem. 

\begin{theorem}[Stone 1975]
\label{thm:unifopt}
If the cost function takes the form $c(x,y)=y$ for all $y\geq 0$ and $x\in X$ and the detection function is regular,  a uniformly optimal search plan $\varphi^\star$ within $\Phi(E)$ can be found for any target distribution $\Pi$ as follows.  Define
\begin{equation}
\label{eq:rate}
q_x (y)=\pi(x)\frac{\partial }{\partial y}d(x, y), \quad x\in X \text{ and $y\geq 0$},
\end{equation}

\begin{equation}
q^{-1}_x(\lambda)=\left\{
		                           \begin{array}{ll}
		                           \text{the inverse function of $q_x(y)$  evaluated at $\lambda$},& \hbox{if $0<\lambda\leq q_x(0)$,} \\
					0, & \hbox{if $\lambda>q_x(0)$,} 
		                          \end{array}
		                         \right.
\end{equation}		                       
and 
\begin{equation}
Q(\lambda)= \left\{
		                           \begin{array}{ll}
		                          \sum_{x\in X} q_x^{-1}(\lambda),  & \hbox{in the discrete case,} \\
					 \int_X q_x^{-1}(\lambda)dx,  & \hbox{in the continuous case.} 
		                          \end{array}
		                         \right.
\end{equation}
Then a uniformly optimal search plan for $\Pi$ within  $\Phi(E)$ is given by $\varphi^\star(x, t)=q_x^{-1}(Q^{-1}(E(t)))$ where $Q^{-1}$ is the inverse function of $Q$. 
\end{theorem}
Here $q_x(y)$ accounts for both the uncertainty of the target location and the future change in detection.  Intuitively,  it represents the ``rate of return'' for allocating more effort to the location $x$ given that $y$ amount of effort has already been allocated there.  Thus,  $q^{-1}_x(\lambda)$ stands for the amount of allocation already required at $x$ so the rate of return for placing more allocation equals $\lambda$, or $0$ if the rate of return is less than $\lambda$ before any allocation at $x$. The idea behind Theorem~\ref{thm:unifopt} is as follows: we distribute search effort by gradually decreasing the rate of return $\lambda$ that is allowed into the allocation; hence,  effort distributed is proportional to the areas with the highest rate of return at the current moment.  We stop when the rate of return $\lambda$ falls to the threshold that corresponds to the event `` total allocation has reached the budget $E$''.  This threshold is quantified by the function $Q$, where $Q(\lambda)$ is the total allocation if every location $x$ in the search space is allocated effort until the rate of return $q_x(y)$ for allocating more effort there falls to $\lambda$. 

\section{Two motivating examples}
Here we first look at two concrete examples,  one discrete and one continuous, to build the intuition for the key results in the next section.

\begin{example}
\label{ex:discrete1}
This example slightly generalizes the example on Page 3 of Stone (1975). Suppose the possibility area $X=\{1, 2\}$, the true target location is cell $1$,  and the search budget at time $t$ is $E(t)$.   We take the target distribution to be $\pi(1)=p$ and $\pi(2)=1-p$, where $1/2<p<1$.  The detection function $d$  is homogeneous and exponential with rate $\alpha=1$, i.e.,  $d(x, y)=1-e^{-y}, \ y\geq 0$ for all $x\in X$.  Recall our convention in Section~2: the cost function always takes the form $c(x, y)=y$ for all $x\in X$.  

Also, the exponential function is a regular detection function.  (Indeed, $d(x, 0)=1-e^{-0}=1$ and $\partial d(x, y)/ \partial y=ye^{-y}$, which is continuous, positive, and strictly increasing in $y$ for all $x\in X$.) Therefore,  the hypothesis of Theorem~\ref{thm:unifopt} is satisfied.  It follows that the uniformly optimal search plan $\varphi^\star$ exists.   By the algorithm in Theorem~\ref{thm:unifopt} or the same argument as on Pages 4--6 of Stone (1975),  we know $\varphi^\star$ is given by 
\[\varphi^\star (1, t)=\left\{
		                           \begin{array}{ll}
		                            E(t),  & \hbox{if $0<E(t)\leq  \ln \left(\frac{p}{1-p}\right) $,} \\
					 \frac{1}{2}\left[E(t)+\ln \left(\frac{p}{1-p}\right) \right],  & \hbox{if $E(t)>\ln \left(\frac{p}{1-p}\right) $,} 
		                          \end{array}
		                         \right.
		                         \]
and 
\[\varphi^\star (2, t)=\left\{
		                           \begin{array}{ll}
		                          0,  & \hbox{if $0< E(t) \leq \ln \left(\frac{p}{1-p}\right) $,} \\
					 \frac{1}{2}\left[E(t)-\ln \left(\frac{p}{1-p}\right) \right],  & \hbox{if $E(t)>\ln \left(\frac{p}{1-p}\right)$.} 
		                          \end{array}
		                         \right.
		                         \]
Suppose $E(t)>\ln [p/(1-p)]$ for all $t>0$ in a real-world search mission. Then the subjective probability of detection equals
\begin{eqnarray*}
P[\varphi^\star(\cdot, t)] &=& \pi(1) d(1,  \varphi^\star (1, E(t)))+\pi(2)d(2, \varphi^\star (2, E(t)))\\
			&=& p\left(1-e^{- \frac{1}{2}\left[E(t)+\ln \left(\frac{p}{1-p}\right) \right]}\right)+(1-p)\left(1-e^{- \frac{1}{2}\left[E(t)-\ln \left(\frac{p}{1-p}\right) \right]}\right) \\
			&=& 1-2\sqrt{p(1-p)}e^{-\frac{1}{2}E(t)}.
\end{eqnarray*}

Since the true target location is cell $1$,  the true probability of detection, as a function of $p$,  is 
\begin{eqnarray*}
P^{\#}[\varphi^\star(\cdot, t)] (p)&=&  d(1,  \varphi^\star (1, E(t)))\\
				&=& 1-e^{- \frac{1}{2}\left[E(t)+\ln \left(\frac{p}{1-p}\right) \right]}\\
				&=& 1-e^{-E(t)/2}\sqrt{\frac{1-p}{p}}\\
				&=& 1-e^{-E(t)/2}\sqrt{\frac{1}{p}-1},				
\end{eqnarray*}
which is increasing in $p$ on $(0, 1)$. Also, it follows from (\ref{eq:truemean}) that the true mean time to find the target $\mu^{\#}(\varphi)(p)$, as a function of $p$, is decreasing in $p$ on $(0, 1)$.  Therefore,  regardless of the value of $p$, there is a $\widetilde{p}\in (0, 1)$ such that $\widetilde{p}>p$,  $P^{\#} [\varphi^\star(\cdot, t)](\widetilde{p})>P^{\#}[\varphi^\star(\cdot, t)](p)$ for all $t>0$ and $\mu^{\#} (\varphi^\star)(\widetilde{p})<\mu^{\#} (\varphi^\star)(p)$.

\end{example}
\noindent \textbf{Remark.} By symmetry,   if the true target location is cell $2$, then there is a $\widetilde{p}\in (0, 1)$ such that $0<\widetilde{p}<p$ and  $P^{\#} [\varphi^\star(\cdot, t)](\widetilde{p})>P^{\#}[\varphi^\star(\cdot, t)](p)$ for all $t>0$ and $\mu^{\#}(\varphi^\star)(\widetilde{p})<\mu^{\#} (\varphi^\star)(p)$.  Therefore, we will have the same conclusion.

\begin{example}
\label{ex:continuous1}
This example slightly generalizes Examples~2.2.1 and 2.2.7 of Stone (1975).
Suppose the possibility area is $X=\mathbb{R}^2$ and the true target location is $x_0$.  Here the search is assumed to be conducted at speed $v$ using a sensor with a sweep width $W$. Thus, if there is time $t$ available for search, then the search budget will be $E(t)=Wvt$. 
The target distribution is circular normal, i.e., 
 \begin{equation}
\pi(x_1, x_2)=\frac{1}{2\pi\sigma^2}e^{-\frac{x_1^2+x_2^2}{2\sigma^2}}, \quad (x_1, x_2)\in X= \mathbb{R}^2,
\end{equation}
and the detection function is nonhomogeneous exponential: $d(x, y)=1-e^{-\alpha_xy}$ for all $x\in X$ and $y\geq 0$ where $\alpha_x>0$. For convenience, we will use polar coordinates.  Then $x_0=(r_0, \theta_0)$ and 
\[
q_{(r, \theta)}^{-1}(\lambda)
=\left\{
		                           \begin{array}{ll}
		                          -\frac{1}{\alpha_x}\left[\ln \left(\frac{2\pi \sigma^2\lambda}{\alpha_x}\right)+\frac{r^2}{2\sigma^2}\right]  & \hbox{if $r^2\leq 2\sigma^2\ln\left(\frac{\alpha_x}{2\pi \sigma^2\lambda}\right)$,} \\
					0,  & \hbox{if $r^2> 2\sigma^2\ln\left(\frac{\alpha_x}{2\pi \sigma^2\lambda}\right)$.} 
		                          \end{array}		                     
		                         \right.
\]
Therefore, 
\begin{eqnarray*}
Q(\lambda) &=& \int_0^{2\pi}\int_0^\infty q_{(r, \theta)}^{-1} (\lambda)r dr d\theta \\
		 &=&-\frac{ 2\pi}{\alpha_x} \int_0^{\sigma\left[2\ln \left(\frac{\alpha_x}{2\pi \sigma^2 \lambda}\right)\right]^{1/2}} 
		         \left[\ln \left(\frac{2\pi \sigma^2\lambda}{\alpha_x}\right)+\frac{r^2}{2\sigma^2}\right] r dr\\
		&=& \frac{\pi \sigma^2}{\alpha_x}\left[\ln\left(\frac{2\pi \sigma^2\lambda}{\alpha_x}\right)\right]^2,
\end{eqnarray*}
and
\[
Q^{-1}(K)=\left\{
		                           \begin{array}{ll}
		                          \frac{\alpha_x}{2\pi \sigma^2}\exp\left[-\left(\frac{\alpha_x K}{\pi \sigma^2}\right)^{1/2}\right],  & \hbox{if $r^2\leq 2\sigma^2\ln\left(\frac{\alpha_x}{2\pi \sigma^2 K}\right)$,} \\
					0,  & \hbox{if $r^2> 2\sigma^2\ln\left(\frac{\alpha_x}{2\pi \sigma^2 K}\right)$.} 
		                          \end{array}
		                         \right.
\]
It follows that the uniformly optimal plan for $\Pi$ within $\Phi(E)$ exists and is given by (in polar coordinates)
\begin{eqnarray*}
\varphi^\star((r, \theta), t) &=& \left\{
		                           \begin{array}{ll}
		                           \frac{1}{\alpha_x}\left[\left(\frac{\alpha_x Wvt}{\pi \sigma^2}\right)^{1/2}-\frac{r^2}{2\sigma^2}\right],& \hbox{if $r^2\leq 2\sigma^2\left(\frac{\alpha_xWvt}{\pi \sigma^2}\right)^{1/2}$,} \\
					0, & \hbox{if $r^2> 2\sigma^2\left(\frac{\alpha_xWvt}{\pi \sigma^2}\right)^{1/2}$,} 
		                          \end{array}
		                         \right.\\
		                         &=& 
		                          \left\{
		                           \begin{array}{ll}
		                           \frac{1}{\alpha_x}\left[H_x\sqrt{t}-\frac{r^2}{2\sigma^2}\right],& \hbox{if $0<r\leq R_x(t)$,} \\
					0, & \hbox{if $r>R_x(t)$,} 
		                          \end{array}
		                         \right.
\end{eqnarray*}		                        	                         
where $R^2_x(t)=2\sigma^2H_x\sqrt{t}$ and $H_x=\sqrt{\alpha_xWv/\pi\sigma^2}$.   In this case, the allocation at a point $x=(r_x, \theta_x)$ depends only on $\alpha_x$ and $r_x$ (not on $\theta_x$) and that the area assigned non-zero allocation is a disc whose radius grows with $t$.  The subjective probability of detection is found to be
\begin{equation}
P[\varphi^\star(\cdot, t)]=1-(1+H_x\sqrt{t})e^{-H_x\sqrt{t}}, \quad t\geq 0.
\end{equation}
Also, the true probability of detection is a function is $\sigma$ and equals
\begin{eqnarray*}
P^{\#}[\varphi^\star(\cdot, t)](\sigma) &=& d(x_0, \varphi^\star(x_0,  t)) \\
&=& \left\{
		                           \begin{array}{ll}
		                             1-e^{-\left(H_{x_0} \sqrt{t}-\frac{r_0^2}{2\sigma^2}\right)},  & \hbox{if $0< r_0\leq R_{x_0}(t) $;} \\
					 0,  & \hbox{if $r_0>R_{x_0}(t)$.} 
		                          \end{array}
		                         \right.
\end{eqnarray*}
Without loss of generality, we may assume $r_0\leq R_{x_0}(t)$.  Put
\[h(\sigma)=H_{x_0}\sqrt{t}-\frac{r_0^2}{2\sigma^2}=\left(\frac{\alpha_{x_0}Wv}{\pi}\right)^{1/2}\frac{\sqrt{t}}{\sigma}-\frac{r^2_0}{2\sigma^2}.\]

Since $r_0\leq R_{x_0}(t)$ if and only if $r^2_0\leq2\sigma^2 H_{x_0}\sqrt{t}$, we know $h(\sigma)\geq 0$. In addition,  $h$ is continuous and  L'Hospital's Rule implies
\[
\lim_{\sigma\rightarrow 0^+} h(\sigma)=\lim_{\sigma\rightarrow 0^+} \frac{1}{\sigma^2}\left[\left(\frac{\alpha_{x_0}Wv}{\pi}\right)^{1/2}\sigma\sqrt{t}-\frac{r_0^2}{2}\right]=+\infty.
\]

Therefore, there is a $\widetilde{\sigma}>0$ such that $0<\widetilde{\sigma}<\sigma$ and $h(\widetilde{\sigma})>h(\sigma)$.  Hence, $P^{\#}[\varphi^\star(\cdot, t)](\widetilde{\sigma})>P^{\#}[\varphi^\star(\cdot, t)](\sigma)$ for all $t>0$,  and $\mu^{\#}(\varphi)(\widetilde{\sigma})<\mu^{\#}(\varphi)(\sigma)$.

\end{example}

\section{Main results}

Suppose the uniformly optimal search plan for $\Pi$ within $\Phi(E)$ exists.  We will show that there is another target distribution $\widetilde{\Pi}$ such that $P^{\#}[\widetilde{\varphi}^\star]>P^{\#}[\varphi^\star]$ where $\widetilde{\varphi}^\star$ is the uniformly optimal search plan for $\widetilde{\Pi}$ and $E(t)$. We remind the reader that the target distribution is always non-degenerate due to the uncertainty of the true target location in the search problem.

\subsection{Misspecified models}

If $x_0\in X$, we say our model is \emph{well-specified}; otherwise, we say our model is \emph{misspecified}.  In practice,  we use all the available information to specify the possibility area $X$, but there is no guarantee that our model is well-specified.  For example, the initial model used in the search of Malaysia Airlines Flight 370 was misspecified due to misformation.  If the target is found during a search, then we know our model is well-specified. But that is hindsight.  In many cases,  there is  no way to be absolutely sure that our model is well-specified before the search begins. Therefore,  the case of a misspecified model deserves consideration.

If our model is misspecified, then the possibility area $X$, being the support of the target distribution $\Pi$, will not contain the target.  Since a search plan only allocates effort on the possibility area $X$,  the true probability of detection will be zero.  Then it is trivial to see that if  $\widetilde{\Pi}$ is a new target distribution whose support $\widetilde{X}$ contain $x_0$ and $\varphi$ is any search plan (based on $\widetilde{\Pi}$) that puts nonzero effort on $x_0$, then we will have a positive true probability of detection using $\widetilde{\Pi}$.  Next, we turn to the nontrivial case, i.e., the well-specified case.

\subsection{Well-specified models}

The above two motivating examples provide some insight into the well-specified case.  If a new target distribution increases the rate of returns at $x_0$ in a way that the total allocation does not decrease,  then the value of $\varphi^\star$ at $x_0$ will increase. This will increase the true probability of detection and shorten the true mean time to find the target. 

\subsubsection{The discrete case}

The detection function in Example~1 is homogeneous over $X$. The next example shows that the same phenomenon can occur even if the detection function is non-homogeneous.

\begin{example}[Generalization of Example~\ref{ex:discrete1}]
\label{ex:counterdiscrete}
Suppose $X=\{1, 2\}$,  {\color{blue}$x_0=1$,} the target density is $\pi(1)=p$ and $\pi(2)=1-p$ where $1/2<p<1$,  and the detection function is $d(x, y)=1-e^{-\alpha_x, y}$ for $x\in X$ and $y\geq 0$ where $\alpha_1>0$, $\alpha_2>0$, and $\alpha_1\neq\alpha_2$.  Then $\frac{\partial }{\partial y}d(x, y)=\alpha_x e^{-\alpha_x y}$ for $y\geq 0$. Thus, 
\begin{eqnarray*}
q_1^{-1}(\lambda)&=&\left\{
		                           \begin{array}{ll}
		                           \frac{1}{\alpha_1}\ln \left(\frac{p\alpha_1}{\lambda}\right),& \hbox{if $0<\lambda\leq p\alpha_1$,} \\
					0, & \hbox{otherwise,}
		                          \end{array}
		                         \right.		\\
q_2^{-1}(\lambda)&=&\left\{
		                           \begin{array}{ll}
		                           \frac{1}{\alpha_2}\ln \left[\frac{(1-p)\alpha_2}{\lambda}\right],& \hbox{if $0<\lambda\leq (1-p)\alpha_2$,} \\
					0, & \hbox{otherwise.}
		                          \end{array}
		                         \right.		
\end{eqnarray*}
We need to consider two cases: (i)~$(1-p)\alpha_2<  p\alpha_1$, and (ii)~$(1-p)\alpha_2\geq  p\alpha_1$.

In Case~(i), we have 

\begin{eqnarray*}
Q(\lambda) &=& q_1^{-1}(\lambda) +q_2^{-1}(\lambda)  \\
		 &=& \left\{
		                           \begin{array}{ll}
		                           \ln\left[\frac{(p\alpha_1)^{1/\alpha_1}((1-p)\alpha_2)^{1/\alpha_2}}{\lambda^{1/\alpha_1+1/\alpha_2}}\right],& \hbox{if $0<\lambda\leq (1-p)\alpha_2$,} \\
					\frac{1}{\alpha_1}\ln \left(\frac{p\alpha_1}{\lambda}\right), & \hbox{$(1-p)\alpha_2<\lambda<p\alpha_1$,} \\
					0, & \hbox{otherwise,}
		                          \end{array}
		                         \right.
\end{eqnarray*}
Hence,

\[		                         
Q^{-1}(K)=\left\{
		                           \begin{array}{ll}
p\alpha_1 e^{-K \alpha_1}, & \hbox{if $0<K \leq\ln \left[\frac{p\alpha_1}{(1-p)\alpha_2}\right]^{1/\alpha_1}$,} \\
  \left[e^{-K} (p\alpha_1)^{1/\alpha_1}((1-p)\alpha_2)^{1/\alpha_2}\right]^{\frac{\alpha_1\alpha_2}{\alpha_1+\alpha_2}},  & \hbox{if $K>\ln \left[\frac{p\alpha_1}{(1-p)\alpha_2}\right]^{1/\alpha_1}$.}
		                          \end{array}
		                         \right.	                         
\]

It follows that the uniformly optimal plan for $\Pi$ within $\Phi(E)$ is given by 
\begin{eqnarray*}
\varphi^\star(1, t)&=&q_1^{-1}(Q^{-1}(E(t))) \\
&=&  \left\{
		                           \begin{array}{ll}
		                           E(t),& \hbox{if $0<E(t)\leq \ln \left[\frac{p\alpha_1}{(1-p)\alpha_2}\right]^{1/\alpha_1}$,} \\
					\frac{1}{\alpha_1+\alpha_2}\left(\alpha_2E(t)+\ln \left[\frac{p\alpha_1}{(1-p)\alpha_2}\right]\right) , & \hbox{if $E(t)>\ln \left[\frac{p\alpha_1}{(1-p)\alpha_2}\right]^{1/\alpha_1}$,}
		                          \end{array}
		                         \right. \\
\varphi^\star(2, t)&=&q_2^{-1}(Q^{-1}(E(t))) \\
&=&  \left\{
		                           \begin{array}{ll}
		                           0,& \hbox{if $0<E(t)\leq \ln \left[\frac{p\alpha_1}{(1-p)\alpha_2}\right]^{1/\alpha_1}$,} \\
					\frac{1}{\alpha_1+\alpha_2}\left(\alpha_1E(t)-\ln \left[\frac{p\alpha_1}{(1-p)\alpha_2}\right]\right) , & \hbox{if $E(t)>\ln \left[\frac{p\alpha_1}{(1-p)\alpha_2}\right]^{1/\alpha_1}$.}
		                          \end{array}
		                         \right. \\		                         
\end{eqnarray*}
First, consider the case $E(t)> \ln \left[\frac{p\alpha_1}{(1-p)\alpha_2}\right]^{1/\alpha_1}$ for all $t>0$. Regardless of whether $x_0=1$ or $x_0=2$,  the same argument as in Example~1 shows that there is a $\widetilde{p}\neq p$ such that $P^{\#}[\widetilde{\varphi}^\star(\cdot, t)]>P^{\#}[\varphi^\star(\cdot, t)]$ and $\mu^{\#}(\widetilde{\varphi}^\star)<\mu^{\#}(\varphi^\star)$, where $\widetilde{\varphi}^\star$ is the uniformly optimal search plan for $\widetilde{\pi}$, and $\widetilde{\pi}(1)=\widetilde{p}$ and $\widetilde{\pi}(2)=1-\widetilde{p}$.   The analysis for Case~(ii) is completely similar and hence is omitted.  

Now consider the case $0<E(t)\leq \ln \left[\frac{p\alpha_1}{(1-p)\alpha_2}\right]^{1/\alpha_1}$ for all $t>0$. In this case,   the true detection probability is a constant for all $t>0$.  Since $p/(1-p)$ is increasing  in $p$ for all $0<p<1$, $\ln \left[\frac{p\alpha_1}{(1-p)\alpha_2}\right]^{1/\alpha_1}$ is increasing in $p$. Thus,   the true detection probability will remain the same when  $p$ increases.

\end{example}

Example 3 shows assigning more target probability to the true target location does not necessarily increase the true detection probability.  However, if the amount of effort is sufficiently large,  then putting more target probability mass will increase the true detection probability.  Part~(ii) of the next theorem establishes this fact in the general setup.

\begin{theorem}
\label{thm:nooptdiscrete1}
Assume $X$ is discrete and $x_0$ is the true cell in which the target is located.  Suppose the cost function takes the form $c(x,y)=y$ for all $y\geq 0$ and $x \in X$,  the detection function is regular,  $\Pi$ is a target distribution on $X$, and $\varphi^\star$ is the uniformly optimal search plan for $\Pi$ within $\Phi(E)$ such that $0\leq \varphi^\star(x_0, t)<E(t)$.  If either of the following conditions holds:
\begin{enumerate}
\item[(i)]there is an $x_1\in X$ such that $x_1\neq x_0$ and $\pi(x_1)>\pi(x_0)$, 
\item[(ii)]$1>\pi(x_0)\geq \pi(x)$ for all $x\neq  x_0$ and  $\ E(t)>q_{x_1}^{-1}(q_{x_2}(0))$ for all $t>0$,
\end{enumerate}
then there exists another target distribution $\widetilde{\Pi}\neq \Pi$ such that $P^{\#}[\widetilde{\varphi}^\star(\cdot, t)]>P^{\#}[\varphi^\star(\cdot, t)]$ for all $t>0$ and $\mu^{\#}(\widetilde{\varphi}^\star)<\mu^{\#}(\varphi^\star)$, where $\widetilde{\varphi}^\star$ is the uniformly optimal search plan for $\widetilde{\Pi}$ within $\Phi(E)$.
\end{theorem}

\begin{proof}
The inverse function of the  rate of return $q_x(y)=\pi(x) \frac{\partial}{\partial y}d(x,y)$ is given by
\begin{equation*}
q^{-1}_x(\lambda)=\left\{
		                           \begin{array}{ll}
		                           \text{the inverse function of $q_x(y)$  evaluated at $\lambda$},& \hbox{if $0<\lambda\leq q_x(0)$,} \\
					0, & \hbox{if $\lambda>q_x(0)$.} 
		                          \end{array}
		                         \right.
\end{equation*}	
Suppose $\widetilde{\pi}$ is the new target density to be constructed.  Since $\frac{\partial}{\partial y}d(x, y)$ is independent of the target density $\pi$,  $\widetilde{\pi}(x_0)>\pi(x_0)$ will imply $\widetilde{q}_{x_0}(y)> q_{x_0}(y)$ for all $y\geq 0$.  On the other hand,  $q_x^{-1}(\lambda)$ equals the inverse of $\ \frac{\partial}{\partial y}d(x, y)$ evaluated at $\frac{\lambda}{\pi(x)}$ when $\lambda\leq q_x(0)$.  Since $\ \frac{\partial}{\partial y}d(x, y)$ is strictly decreasing, we have
\begin{equation}
\label{eq:discrete1}
\widetilde{q}^{-1}_{x_0}(\lambda)\left\{
		                           \begin{array}{ll}
		                           >q_{x_0}^{-1}(\lambda),& \hbox{for $0<\lambda\leq \widetilde{q}_{x_0}(0)$,} \\
					=q_{x_0}^{-1}(\lambda)=0, & \hbox{if $\lambda>\widetilde{q}_{x_0}(0)$.} 
		                          \end{array}
		                         \right.
\end{equation}	
Without loss of generality, we may assume $\widetilde{Q}^{-1}(E(t))\leq \widetilde{q}_{x_0}(0)$.
We are going to construct a new target distribution $\widetilde{\Pi}$ such that
\begin{equation}
\label{eq:condition}
\widetilde{q}_{x_0}^{-1}(\widetilde{Q}^{-1}(E(t)))>q_{x_0}^{-1}(Q^{-1}(E(t))),
\end{equation}
which will imply $P^{\#}[\widetilde{\varphi}^\star(\cdot, t)]>P^{\#}[\varphi^\star(\cdot, t)]$ for all $t>0$ and $\mu^{\#}(\widetilde{\varphi}^\star)<\mu^{\#}(\varphi^\star)$ via (\ref{eq:trueprob}) and (\ref{eq:truemean}). To this end,  consider two cases: (i)~there is an $x_1\in X$ such that $x_1\neq x_0$ and $\pi(x_1)>\pi(x_0)$ and (ii)~$1>\pi(x_0)\geq \pi(x)$ for all $x\neq  x_0$ and $E(t)>q_{x_1}^{-1}(q_{x_2}(0))$ for all $t>0$.

In Case~(i), we construct the new target density $\widetilde{\pi}$ by switching the probability masses of $\pi$ on $x_0$ and $x_1$ and keep everything else intact. That is, we define
\[ \widetilde{\pi}(x)=\left\{
		                           \begin{array}{ll}
		                           \pi(x_1),& \hbox{if $x=x_0$,} \\
					\pi(x_0), & \hbox{if $x=x_1$,} \\
					\pi(x), & \hbox{otherwise.}
		                          \end{array}
		                         \right.		                         
		                         \]
Then $\widetilde{\pi}(x_0)>\widetilde{\pi}(x_1)$.  By (\ref{eq:discrete1}), we have $\widetilde{q}^{-1}_{x_0}(\lambda)>q^{-1}_{x_0}(\lambda)$ for $0<\lambda\leq \widetilde{q}_{x_0}(0)$.  Also,   the definition of $\widetilde{\pi}$ and (\ref{eq:rate}) imply 
\[\widetilde{Q}(\lambda)=\sum_{x\in X}\widetilde{q}_x^{-1}(\lambda)=\sum_{x\in X}q_x^{-1}(\lambda)=Q(\lambda), \quad \text{for all $\lambda>0$},\]
which further implies $\widetilde{Q}^{-1}(K)=Q^{-1}(K)$ for all  $K>0$. Therefore, 
\[\widetilde{q}^{-1}_{x_0}(\widetilde{Q}^{-1}(E(t)))=\widetilde{q}^{-1}_{x_0}(Q^{-1}(E(t)))>q_{x_0}^{-1}(Q^{-1}(E(t))).\]

In Case~(ii), we may assume without loss of generality that $X=\{1, \ldots, m\}$ for some integer $m\geq 2$, $x_0=1$, and $1>\pi(1)>\pi(2)\geq \pi(3)\geq \ldots \geq \pi(m)>0$.  Then the definition of $q_x$ implies $q_1(0)>q_2(0)\geq q_3(0)\ldots\geq q_m(0)$. For $i=1, \ldots, m$, we have
\begin{equation*}
q^{-1}_i(\lambda)=\left\{
		                           \begin{array}{ll}
		                           \left(\frac{\partial}{\partial y}d(i, y)\right)^{-1}\bigg |_{y=\frac{\lambda}{\pi(i)}},& \hbox{if $0<\lambda\leq q_x(0)$,} \\
					0, & \hbox{if $\lambda>q_i(0)$.} 
		                          \end{array}
		                         \right.
\end{equation*}	
Thus, 
\[
Q(\lambda)=\sum_{i=1}^m q_i^{-1}(\lambda)=\left\{
		                           \begin{array}{ll}
		                           \sum_{i=1}^m \left(\frac{\partial}{\partial y}d(i, y)\right)^{-1}\bigg |_{y=\frac{\lambda}{\pi(i)}}, & \hbox{if $0<\lambda\leq q_m(0)$,} \\
					\sum_{i=1}^{m-1} \left(\frac{\partial}{\partial y}d(i, y)\right)^{-1}\bigg |_{y=\frac{\lambda}{\pi(i)}}, & \hbox{if $q_m(0)<\lambda\leq q_{m-1}(0)$,} \\
						\ldots & \hbox{$\ldots$,} \\
 \left(\frac{\partial}{\partial y}d(1, y)\right)^{-1}\bigg |_{y=\frac{\lambda}{\pi(1)}} + \left(\frac{\partial}{\partial y}d(2, y)\right)^{-1}\bigg |_{y=\frac{\lambda}{\pi(2)}}, & \hbox{if $q_3(0)<\lambda\leq q_2(0)$,} \\
 \left(\frac{\partial}{\partial y}d(1, y)\right)^{-1}\bigg |_{y=\frac{\lambda}{\pi(1)}}, & \hbox{if $q_2(0)<\lambda\leq q_1(0)$,} \\
                 0, & \hbox{if $\lambda>q_1(0)$.} 
		                          \end{array}
		                         \right.
\]
It follows that $Q^{-1}$ has the form
\[
Q^{-1}(K)=\left\{
		                           \begin{array}{ll}
		                          q_1(K),& \hbox{if $0<K \leq q_1^{-1}(q_2(0))$,} \\
		                         (q_1^{-1}+q_2^{-1})^{-1}(K),& \hbox{if $q_1^{-1}(q_2(0))<K\leq (q_1^{-1}+q_2^{-1})(q_3(0))$,} \\
		                          \ldots,& \hbox{if $\ldots$},
		                          \end{array}
		                         \right.
\]
and $\varphi^\star$ is of the form
\begin{eqnarray*}
\varphi^\star(1, t) &=& q_1^{-1}(Q^{-1}(E(t)))\\
&=&\left\{
		                           \begin{array}{ll}
		                         E(t),& \hbox{if $0<E(t) \leq q_1^{-1}(q_2(0))$,} \\		    
		                         q_1^{-1}((q_1^{-1}+q_2^{-1})^{-1}(E(t))), & \hbox{if $q_1^{-1}(q_2(0))< E(t) \leq (q_1^{-1}+q_2^{-1})(q_3(0))$,} \\	                 
		                          \ldots,& \hbox{if $\ldots$},
		                          \end{array}
		                         \right.
\end{eqnarray*}
Consider the case $q_1^{-1}(q_2(0))< E(t) \leq (q_1^{-1}+q_2^{-1})(q_3(0))$. By the definition of $q_x^{-1}$, we know 
$\varphi^\star$ is continuous in both $\pi(1)$ and $\pi(2)$ and 
\[
\lim_{\pi(1)\rightarrow 1^-} \varphi^\star(1, t)=\lim_{\pi(1)\rightarrow 1^-}q_1^{-1}((q_1^{-1}+q_2^{-1})^{-1}(E(t)))=q_1^{-1}((q_1^{-1})^{-1}(E(t)))=E(t).
\]
For all other cases where $E(t)>(q_1^{-1}+q_2^{-1})(q_3(0))$,  a similar argument leads to the same conclusion. 
Therefore, there is another target density function $\widetilde{\pi}$ such that (\ref{eq:condition}) holds as far as $E(t)>q_1^{-1}(q_2(0))$ for all $t>0$.

\end{proof}
\noindent \textbf{Remark.} The idea of the above proof is as follows.  In Case~(i), $x_0$ has received less target probability mass than $x_1$, then we can simply switch their target probability masses and keep everything else the same; this will increase the rate of return at $x_0$, but it will not change the total allocation.  As a result, the amount of effort at $x_0$ will increase and hence the true probability of detection will go up too.  In Case~(ii),  $x_0$ has already received the highest prior probability mass.  In this case,  increasing the target probability at $x_0$ will not necessarily increase the true detection probability because when $E(t)\leq q_1^{-1}(q_2(0))$ the uniformly optimal search plan will allocate all effort to $x_0$. However, if $E(t)> q_1^{-1}(q_2(0))$,  assigning more target probability to $x_0$ does lead to an increased true detection probability.  As we will see shortly,  the condition ``$E(t)> q_1^{-1}(q_2(0))$'' is unique to the discrete case.  The reason is that, in the continuous case,  the uniformly optimal search plan will never put all target probability mass on a single point.

\begin{cor}
\label{cor:nooptdiscrete1}
Assume $X$ is discrete and $x_0$ is the true cell in which the target is located.  Suppose the cost function takes the form $c(x,y)=y$ for all $y\geq 0$ and $x \in X$,  the detection function is  regular,  $\Pi$ is a non-degenerate target distribution on $X$ (i.e., $\pi(x)<1$ for all $x\in X$), and $\varphi^\star$ is the uniformly optimal search plan for $\Pi$ within $\Phi(E)$ such that $0\leq \varphi^\star(x_0, t)<E(t)$.  Then there  is a point $x\in X$ such that if $x_0=x$ then there exists another target distribution $\widetilde{\Pi}\neq \Pi$ such that $P^{\#}[\widetilde{\varphi}^\star(\cdot, t)]>P^{\#}[\varphi^\star(\cdot, t)]$ for all $t>0$ and $\mu^{\#}(\widetilde{\varphi}^\star)<\mu^{\#}(\varphi^\star)$, where $\widetilde{\varphi}^\star$ is the uniformly optimal search plan for $\widetilde{\Pi}$ within $\Phi(E)$.
\end{cor}

\subsubsection{The continuous case}

The detection function in Example~2 is homogeneous over $X$. The next example shows that the same conclusion may hold if the detection function is non-homogeneous.

\begin{example}
\label{ex:countercontinuous}
Suppose $a<0<b$, $(a, b)\subseteq X\subseteq \mathbb{R}$, the target distribution is 
\[
\pi (x)=\left\{
		                           \begin{array}{ll}
		                         \frac{1}{b-a},& \hbox{if $a<x<b$,} \\
					0,  & \hbox{otherwise,} 
		                          \end{array}
		                         \right.
\]
and the detection function is 
\[d(x, y)=1-e^{-\alpha_xy},\quad y\geq 0,\]
where 
\[
\alpha_x=\left\{
		                           \begin{array}{ll}
		                           \beta_1,& \hbox{if $x\geq 0$,} \\
					\beta_2,  & \hbox{if $x<0$,} 
		                          \end{array}
		                         \right.
\]
and $\beta_1> \beta_2>0$.  Then 
\[
q_x(y)=\left\{
		                           \begin{array}{ll}
		                          \frac{\beta_1}{b-a}e^{-\beta_1 y},& \hbox{if $0\leq x<b$,} \\
		                          \frac{\beta_2}{b-a}e^{-\beta_2 y},& \hbox{if $a<x<0$,} \\
		                          0,& \hbox{otherwise,} \\
		                          \end{array}
		                         \right.
\]
and
\[
q_x^{-1}(\lambda)=\left\{
		                           \begin{array}{ll}
		                          \frac{1}{\beta_1} \ln\left[\frac{\beta_1}{\lambda(b-a)}\right],& \hbox{if  $0<\lambda\leq \frac{\beta_1}{b-a}$ and $0\leq x<b$,} \\
		                          \frac{1}{\beta_2} \ln\left[\frac{\beta_2}{\lambda(b-a)}\right],& \hbox{if  $0<\lambda\leq \frac{\beta_2}{b-a}$ and $a<x<0$,} \\
		                          0,& \hbox{otherwise.} \\
		                          \end{array}
		                         \right.
\]
It follows that 

\begin{eqnarray*}
Q(\lambda) =\left\{
		                           \begin{array}{ll}
		                         \ln \left(\frac{\beta_1^{b/\beta_1}\beta_2^{-a/\beta_2}}{\left[\lambda (b-a)\right]^{b/\beta_1-a/\beta_2}}\right),& \hbox{if $0<\lambda<\frac{\beta_2}{b-a}$,} \\
		                          \ln\left[\frac{\beta_1}{\lambda(b-a)}\right]^{b/\beta_1},& \hbox{if $\frac{\beta_2}{b-a}<\lambda<\frac{\beta_1}{b-a}$,} \\
		                          0,& \hbox{otherwise,} \\
		                          \end{array}
		                         \right.
\end{eqnarray*}

and

\begin{eqnarray*}
Q^{-1}(K) =\left\{
		                           \begin{array}{ll}
		                          \frac{\beta_1}{b-a}e^{-K\beta_1/b},& \hbox{if $0<K\leq \frac{b}{\beta_1}\ln (\beta_1/\beta_2)$,} \\
		                        \frac{1}{b-a}\left(e^{-K} \beta_1^{b/\beta_1}\beta_2^{-a/\beta_2}\right)^{\frac{1}{b/\beta_1-a/\beta_2}},& \hbox{if $K> \frac{b}{\beta_1}\ln (\beta_1/\beta_2)$.} \\
		                          \end{array}
		                         \right.
\end{eqnarray*}

We have $\frac{b_2}{b-a}\leq Q^{-1}(E(t))<\frac{b_1}{b-a}$ for $0<E(t)\leq \frac{b}{\beta_1}\ln(\beta_1/\beta_2)$ and $Q^{-1}(E(t))\leq \frac{b_2}{b-a}$ for $E(t)> \frac{b}{\beta_1}\ln(\beta_1/\beta_2)$.
Therefore,  the uniformly optimal search plan for $\Pi$ within $\Phi(E)$ is given by 

\begin{eqnarray*}
 & & \varphi^\star (x, t) = q_x^{-1}(Q^{-1}(E(t))) \\
				&=& \left\{
		                           \begin{array}{ll}
		                           \frac{1}{b}E(t) &  \hbox{if  $0<E(t)\leq \frac{b}{\beta_1}\ln(\beta_1/\beta_2)$ and $0\leq x<b$,} \\
		                          \frac{1}{\beta_1} \left[\frac{E(t)}{b/\beta_1-a/\beta_2}-\frac{a/\beta_2}{b/\beta_1-a/\beta_2}\ln \beta_1+
		                          \frac{a/\beta_2}{b/\beta_1-a/\beta_2}\ln \beta_2
		                          \right],& \hbox{if  $E(t)>\frac{b}{\beta_1}\ln(\beta_1/\beta_2)$ and $0\leq x<b$,} \\		                        
		                           \frac{1}{\beta_2} \left[\frac{E(t)}{b/\beta_1-a/\beta_2}-\frac{b/\beta_1}{b/\beta_1-a/\beta_2}\ln \beta_1+
		                          \frac{b/\beta_1}{b/\beta_1-a/\beta_2}\ln \beta_2
		                          \right],& \hbox{if  $E(t)>\frac{b}{\beta_1}\ln(\beta_1/\beta_2)$ and $a<x<0$,} \\
		                          0,& \hbox{otherwise.} \\
		                          \end{array}
		                         \right.
\end{eqnarray*}

Take $\beta_1=3\beta_2$ and $\beta_2>0$.  Assume $x_0=0$ and $E(t)>\frac{b}{\beta_1}\ln 3$. Then we have 

\begin{eqnarray*}
\varphi^\star(x_0, t) &=&    \frac{1}{3\beta_2} \left[\frac{E(t)3\beta_2}{b-3a}-\frac{3a}{b-3a}\ln (3\beta_2)+
		                          \frac{3a}{b-3a}\ln \beta_2
		                          \right]\\
		                          &=& \frac{(E(t)\beta_2-a\ln 3)}{(b-3a)\beta_2}.
\end{eqnarray*}

Consider the following new target distribution
\[
\widetilde{\pi}(x)=\left\{
		                           \begin{array}{ll}
		                         \frac{1}{\widetilde{b}-a},& \hbox{if $a<x<\widetilde{b}$,} \\
					0,  & \hbox{otherwise.} 
		                          \end{array}
		                         \right.
\]
Let $\widetilde{\varphi}^\star$ be the uniformly optimal search plan for $\widetilde{\Pi}$ within $\Phi(E)$.  We can always choose $\widetilde{b}>0$ sufficiently small such that $\widetilde{\varphi}^\star(x_0, t) >\varphi^\star(x_0, t)$; hence $P^{\#}[\widetilde{\varphi}^\star(\cdot, t)]>P^{\#}[\varphi^\star(\cdot, t)]$ for all $t>0$ and $\mu^{\#}(\widetilde{\varphi}^\star)<\mu^{\#}(\varphi^\star)$.
\end{example}

Example~4 suggests a continuous analogue of Theorem~\ref{thm:nooptdiscrete1}.  The next theorem shows that this is indeed the case.

\begin{theorem}
\label{thm:nooptcontinuous1}
Assume $X$ is continuous and $x_0$ is the true target location.  Suppose the cost function takes the form $c(x,y)=y$ for all $y\geq 0$ and $x \in X$,  the detection function is  regular,  $\Pi$ is a target distribution, and $\varphi^\star$ is the uniformly optimal search plan for $\Pi$ within $\Phi(E)$.  Then there exists a target distribution $\widetilde{\Pi}\neq \Pi$ such that $P^{\#}[\widetilde{\varphi}^\star(\cdot, t)]>P^{\#}[\varphi^\star(\cdot, t)]$ for all $t>0$ and $\mu^{\#}(\widetilde{\varphi}^\star)<\mu^{\#}(\varphi^\star)$, where $\widetilde{\varphi}^\star$ is the uniformly optimal search plan for $\widetilde{\Pi}$ within $\Phi(E)$.
\end{theorem}

\begin{proof}
We have $q_x(y)=\pi(x)\frac{\partial }{\partial y}d(x, y)$.  Hence, 

\[
q^{-1}_x(\lambda)=\left\{
		                           \begin{array}{ll}
		                           \left(\frac{\partial }{\partial y}d(x, y)\right)^{-1}\bigg |_{y=\frac{\lambda}{\pi(x)}},& \hbox{if $0<\lambda\leq q_x(0)$,} \\
					0, & \hbox{if $\lambda>q_x(0)$,} 
		                          \end{array}
		                         \right.
\]
Therefore, 
\[
\varphi^\star(x, t)=q_x^{-1}(Q^{-1}(E(t)))=\left\{
		                           \begin{array}{ll}
		                           \left(\frac{\partial }{\partial y}d(x, y)\right)^{-1}\bigg |_{y=\frac{Q^{-1}(E(t))}{\pi(x)}},& \hbox{if $0<Q^{-1}(E(t)) \leq q_x(0)$,} \\
					0, & \hbox{if $Q^{-1}(E(t))>q_x(0)$,} 
		                          \end{array}
		                         \right.
\]
Since both $q_x^{-1}$ and $Q^{-1}$ are continuous,  the last display shows that $\varphi^\star$ is continuous in $\pi(x)$ for each $x\in X$ and all $t>0$. The regularity of $d$ implies that $\frac{\partial }{\partial y}d(x, y)$ is strictly decreasing on $[0, \infty)$ for each $x\in X$.  Thus,  $\left(\frac{\partial }{\partial y}d(x, y)\right)^{-1}$ is strictly decreasing on $(0, q_x(0)/\pi(x)]$ for each $x\in X$.  It follows that if $1>\widetilde{\pi}(x_0)>\pi(x_0)$, then $\widetilde{\varphi}^\star(x_0, t)>\varphi^\star(x_0, t)$ for all $t>0$.  Therefore,  $P^{\#}[\widetilde{\varphi}^\star(\cdot, t)]>P^{\#}[\varphi^\star(\cdot, t)]$ for all $t>0$ and $\mu^{\#}(\widetilde{\varphi}^\star)<\mu^{\#}(\varphi^\star)$.
\end{proof}

\section{Concluding remarks}

The optimal search theory has made remarkable progress since WWII.  Some people believe all interesting problems in the optimal search theory have been solved or reduced to simple numerical problems.  This article dispels such a belief.  As we have seen,  the current ``best practice'' in the search community does not necessarily generate a search plan that maximizes the true detection probability at each moment.  The problem of finding a search plan that maximizes the true detection probability at each moment remains one of the most important unsolved problems.  If this problem is solvable,   we need to find a solution.  If it is unsolvable,  we need to establish this fact.  In that case,  we also need to ask whether the uniformly optimal search plan is the best we can do or maybe there is still room for improvement. Either way,  this problem warrants our future efforts.

The key results derived in this paper might seem theoretical to some people. However, they have enormous societal and practical relevance on several accounts.  First,  a better solution to the search problem can increase the chance of survival for the personnel involved in a civil or military accident.  Second,  even if no survivor is expected in such an incident (e.g., the MH370 flight),  finding the target is crucial for us to give closure to the families of the victims.   Third,  finding the target enables us to learn from past mistakes.  Therefore,  this article communicates an important message to researchers and practitioners in the search and rescue community: we may still be able to improve the current ``best practice''.

In this article, we have focused only on the case of a stationary target.  For a moving target,  the target distribution will be the law of a stochastic process (i.e., a family of distributions), but it is subjective too.  Therefore,  the key results in this article are expected to generalize to the case of a moving target. The same can be said for the multiple target tracking problem (e.g.,  Stone et al. 2014b).

\section*{Acknowledgments}
I thank all those who have provided me with comments on the previous versions.  This project was partially supported by the Faculty Start-Up Funding Program at the University of Texas at Dallas.

\section*{Conflict of interest}
The author has no conflict of interest to declare.


\section*{References}
\begin{description}

\item{} Alpern, S.~and Gal, S.~(2003). \emph{The Theory of Search games and Rendezvous}.  Kluwer: Boston. 

\item{} Alpern, S., Chleboun, P., Katslkas, S.~and Lin, K.Y.~(2021).  Adversarial patrolling in a uniform. \emph{Operations Research}~70(1), 129-140.

\item{} Arkin, V.I.~(1964). Uniformly optimal strategies in search problems. \emph{Theory of Probability and its Applications}~9(4),  647--677. 

\item{} Bourque, FA.~(2019).  Solving the moving target search problem using indistinguishable searchers. \emph{European Journal of Operational Research}~275, 45--52.

\item{} Clarkson, J., Glazebrook, K.D.~and Lin, K.Y.~(2020). Fast or slow: search in discrete locations with two search models. \emph{Operations Research}~68(2), 552--571.

\item{} Everett, H.~(1963). Generalized Lagrange multiplier method for solving probglems of optimum allocation of resources. \emph{Operations Research}~11, 399--417.

\item{} Hong, L.~(2024).  On several properties of uniformly optimal search plans. \emph{Military Operations Research}~29(2), 95---106. 

\item{} Kadane, J.B.~(2015). Optimal discrete search with technological choice.  \emph{Mathematical Methods of Operations Research}~81, 317--336. 

\item{} Kallenberg, O.~(2002). \emph{Foundations of Modern Probability}, Second Edition.  Springer: New York.

\item{} Koopman, B.O.~(1946). Search an screening. \emph{Operations Evaluation Group Report No. 56 (unclassified).} Center for Naval Analysis, Rosslyn, Virginia. 

\item{} Koopman, B.O.~(1956a). The theory of search, I. Kinematic bases. \emph{Operations Research}~4, 324--346.

\item{} Koopman, B.O.~(1956b). The theory of search, II. Target detection. \emph{Operations Research}~4, 503--531.

\item{} Koopman, B.O.~(1956a). The theory of search, III The optimum distribution of searching efforts. \emph{Operations Research}~4, 613--626. 

\item{} Kratzke, T.M., Stone, L.D., and Frost J.R.~(2010). Search and rescue optimal planning system.   \emph{Proceedings of the 13th International Conference on Information Fusion, Edinburgh, UK, July 2010}, 26--29.

\item{} Lidbetter, T.~(2020). Search and rescue in the face of uncertain threats. \emph{European Journal of Operations Research}~285, 1153--1160.

\item{} Lin, K.Y.~(2021). Optimal patrol of a perimeter. \emph{Operations Research}~70(5), 2860--2866.

\item{} Richardson, H.R.~and Stone, L.D.~(1971). Operations analysis during the underwater search for Scorpion. \emph{Naval Research Logistic Quarterly}~18, 141--157. 

\item{} Richardson, H.R., Wagner D.H.~and Discenza, J.H.~(1980). The United States Coast Guard Computer-assisted Search Planning System (CASP). \emph{Naval Research Logistic Quarterly}~27,  659--680. 

\item{} Soza Co. Ltd and U.S. Coast Guard~(1996).  \emph{The Theory of Search: A Simplified Explanation}.  U.S. Coast Guard: Washington, D.C..

\item{} Stone, L.D.~(1973). Totally optimality of incrementally optimal allocations. \emph{Naval Research Logistics Quarterly}~20, 419--430.

\item{} Stone, L.D.~(1975). \emph{Theory of Optimal Search}.  Academic Press: New York.

\item{} Stone, L.D.~(1976). Incremental and total optimization of separable functionals with constraints. \emph{SIAm Journal on Control and Optimization}~14,  791--802.

\item{} Stone, L.D., Royset, J.O., and Washburn, A.R.~(2016). \emph{Optimal Search for Moving Targets}. Springer: New York.

\item{} Stone, L.D.~and Stanshine J.A.~(1971). Optimal searching using uninterrupted contact investigation.  \emph{SIAM Journal of Applied Mathematics}~20, 241--263.

\item{} Stone, L.D.~(1992). Search for the SS Central America: mathematical treasure hunting. \emph{Interfaces}~22: 32--54. 

\item{} Stone, L.D., Keller, C.M. , Kratzke, t.M.~and Strumpfer, J.P.~(2014a). Search for the wreckage of Air France AF 447. \emph{Statistical Science}~29:69--80. 

\item{} Stone, L.D.,  Streit, R. L.,  Corwin, T.L.~and Bell, K.L.~(2014b). \emph{Bayesian Multiple Target Tracking}, Second Edition. 
Artech House: Boston,


\item{} Vermeulen, J.F.~and Brink, M.V.~(2017).   The search for an altered moving target.  \emph{Journal of the Operational Research Society}~56(5), 514--525.

\item{} Washburn, A.~(2014). \emph{Search and Detection}, 5th Edition. Create Space: North Carolina.

\end{description}

\end{document}